\documentclass[11pt, reqno]{amsart}
\usepackage{amsmath, amssymb, amsthm, enumerate, mathrsfs, mathtools, hyperref}

\newtheorem{rem}{Remark}

\newtheorem{thm}{Theorem}
\newtheorem{lem}{Lemma}
\newtheorem{cor}{Corollary}

\newtheorem{assumptionA}{A-\hspace{-1.2mm}}
\newtheorem{assumptionAB}{AB-\hspace{-1.2mm}}
\newtheorem{assumptionB}{B-\hspace{-1.2mm}}

\title[Explicit Approximations for L\'evy Driven SDEs]{On Explicit Approximations for L\'evy Driven SDEs with Super-linear Diffusion Coefficients}
\author[C. Kumar and S. Sabanis]{Chaman Kumar and Sotirios Sabanis \\ School of Mathematics \\ The University of Edinburgh \\\ United Kingdom}

\begin{document}
\maketitle
\begin{abstract}
Motivated by the results of \cite{sabanis2015}, we propose explicit Euler-type schemes for SDEs with random coefficients driven by L\'evy noise when the drift and diffusion coefficients can grow super-linearly. As an application of our results, one can construct explicit Euler-type schemes for SDEs with delays (SDDEs) which are driven by L\'evy noise and have super-linear coefficients. Strong convergence results are established and their rate of convergence is shown to be equal to that of the classical Euler scheme. It is proved that the optimal rate of convergence is achieved for $\mathcal{L}^2$-convergence which is consistent with the corresponding results available in the literature.
\end{abstract}
%---------------------------------------------------------------------------------------
\section{Introduction}
%---------------------------------------------------------------------------------------
Let $(\Omega, \{\mathscr{F}\}_{t \geq 0}, \mathscr{F}, P)$ be a filtered probability space satisfying the usual conditions. Let $w$ be $\mathbb{R}^m$-valued standard Wiener process and $N(dt,dz)$ be a Poisson random measure defined on $\sigma$-finite measure space $(Z,\mathscr{Z}, \nu)$ with intensity measure $\nu \not \equiv 0$ (for the case when $\nu \equiv 0$, readers can refer to \cite{sabanis2015}). Set $\tilde N(dt, dz):=N(dt, dz)-\nu(dz)dt$.

Let $b_t(x)$ and $\sigma_t(x)$ be $\mathscr{P}\otimes\mathscr{B}(\mathbb{R}^d)$-measurable functions  in $\mathbb{R}^d$ and $\mathbb{R}^{d \times m}$ respectively. Also, let $\gamma_t(x,z)$ be a $\mathscr{P}\otimes \mathscr{B}(\mathbb{R}^d)\otimes \mathscr{Z}$-measurable function in $\mathbb{R}^d$. Let $T>0$ be a constant and we fix $t_0$ and $t_1$  satisfying $0\leq t_0<t_1\leq T$.  In this article, we consider the following SDE,
\begin{align} \label{eq:sderc}
dx_t=b_t(x_t)dt+\sigma_t(x_t)dw_t+\int_{Z} \gamma_t(x_t,z) \tilde N(dt, dz)
\end{align}
almost surely for any $t\in [t_0, t_1]$ with initial value as an $\mathscr{F}_{t_0}$-measurable random variable $x_{t_0}$ in $\mathbb{R}^d$.
\begin{rem}
We use $x_t$ instead of $x_{t-}$ on the right hand side of the equation \eqref{eq:sderc} for notational convenience that  shall be used throughout this article. Moreover, this does not cause any problem because the compensators of the martingales driving the equation are continuous.
\end{rem}
%--------------------------------------------------------
For every $n \in \mathbb{N}$, suppose that the functions $b_t^n(x)$ and $\sigma_t^n(x)$ are $\mathscr{P}\otimes\mathscr{B}(\mathbb{R}^d)$-measurable and take values in $\mathbb{R}^d$ and $\mathbb{R}^{d\times m}$ respectively. Furthermore, let the function $\gamma_t^n(x,z)$ be $\mathscr{P}\otimes \mathscr{B}(\mathbb{R}^d) \otimes \mathscr{Z}$-measurable with values in $\mathbb{R}^d$ for every $n \in \mathbb{N}$.  In this article, we propose an explicit Euler-type  scheme defined below. For every $n \in \mathbb{N}$,
\begin{align} \label{eq:esrc}
dx_t^n=b_t^n(x_{\kappa(n,t)}^n)dt+\sigma_t^n(x_{\kappa(n,t)}^n)dw_t+\int_Z \gamma_t^n(x_{\kappa(n,t)}^n, z) \tilde{N}(dt,dz)
\end{align}
almost surely for any $t \in [t_0,t_1]$ with initial value as an $\mathscr{F}_{t_0}$-measurable random variable $x_{t_0}^n$ in $\mathbb{R}^d$. Also,  the function $\kappa(n,t)$ is given by $\kappa(n,t):=\lfloor n(t-t_0)\rfloor/n+t_0$ for any $t \in [t_0, t_1]$.
\newline \newline
%--------------------------------------------------------
The SDEs of type \eqref{eq:sderc} are popular models in finance, economics, engineering, ecology, medical sciences  and many other areas where problems are influenced by event-driven uncertainties. Often, such SDEs do not possess any explicit solution and one has to resort to numerical schemes to obtain their approximate solutions. Details of explicit and implicit schemes for SDEs driven by L\'evy noise can be found in \cite{platen2010} and the references therein. It is well known that the moments of the classical Euler scheme of SDE \eqref{eq:sderc} may diverge to infinity in finite time when the coefficients of the SDE grow super-linearly -  \cite{hutz2010} proved this result for SDEs with continuous paths. For  SDEs with super-linear coefficients, implicit schemes can be used to obtain their approximate solutions, but they are typically computationally very demanding. In recent years, the focus has been shifted to the development of efficient, explicit numerical schemes with optimal rates of convergence and a stream of research articles has appeared in the literature which reported significant progress in this direction. For continuous SDEs, one can refer to \cite{hutz2012, hutz2014, hutz2015, kumar2016b, sabanis2013, sabanis2015, tret2013} and the references therein, whereas for SDEs driven by L\'evy noise, one can refer to \cite{dare2016, kumar2016a}. Moreover, new results appeared in the direction of non-polynomial lower error bounds for approximations of nonlinear SDEs, see \cite{JM-GY2016, Yaros2016}.

In this article, we propose an explicit Euler-type scheme \eqref{eq:esrc} of SDE \eqref{eq:sderc} where both drift and diffusion coefficients are allowed to grow superlinearly, whereas the jump coefficient can grow linearly. The strong convergence is established and the rate of convergence is shown to be equal to that of the classical Euler scheme. To the best of the authors' knowledge, these are the first such results in the literature for L\'evy driven SDEs. 

Further, the techniques discussed in this article and in \cite{kumar2016a, kumar2016b} can be combined to develop explicit Milstein-type and higher-order schemes which converge to SDEs \eqref{eq:sderc} with super-linear drift and diffusion coefficients in the strong sense, however this is not the focus of the current article.  Finally, by adopting the approach of \cite{gyongy2012, kumar2014}, the results obtained here can also be extended to the case of delay equations (SDDEs) as illustrated in Section \ref{SDDE section} below.

To conclude this section, let us introduce some basic notation. We use $|x|$ to denote the Euclidean norm of $x\in \mathbb{R}^d$, $|\sigma|$ and $\sigma^*$ to denote the Hilbert-Schmidt norm  and the transpose of $\sigma \in \mathbb{R}^{d \times m}$ respectively. For any  $x,y \in \mathbb{R}^d$, $xy$ stands for their inner product. $I_A$ stands for the indicator function of a set $A$ and $\lfloor x \rfloor$ for the integer part of a real number $x$. For an $\mathbb{R}^d$-valued random variable $X$, $X \in L^p(\Omega)$ means $E|X|^p<\infty$ and for a sequence $\{X^n\}_{n \in \mathbb{N}}$ of $\mathbb{R}^d$-valued random variables, $\{X^n\}_{n \in \mathbb{N}} \in l_\infty(L^p(\Omega))$ means $\sup_{n \in \mathbb{N}}E|X^n|^p<\infty$.  $\mathscr{B}(V)$ denotes the Borel sigma-algebra of a topological space $V$. $\mathscr{P}$ is the predictable sigma-algebra on $\Omega \times \mathbb{R}_{+}$.   Throughout this article, $K>0$ denotes a generic constant that  varies from place to place.

\section{Assumptions and Description of Results}
%-----------------------------------------------------------
We fix $p_0 \geq 2$ and make the following assumptions for SDE \eqref{eq:sderc}. For every $R>0$, consider $C(R)$ which is an $\mathscr{F}_{t_0}$-measurable random variable such that
\begin{align} \notag
\lim_{R \to \infty} P(C(R)>f(R))=0
\end{align}
for a non-decreasing function $f:\mathbb{R}_+ \to \mathbb{R}_+$.
%-----------------------------------------------------------
\begin{assumptionA} \label{as:sderc:initial:value}
$x_{t_0} \in L^{p_0}(\Omega)$.
\end{assumptionA}
%------------------------------------------------------
\begin{assumptionA} \label{as:sderc:coercivity}
%-------------------------------------------------------
There exist a constant $L>0$ and an $\mathscr{F}_{t_0}$-measurable random variable $M \in L^\frac{p_0}{2}(\Omega)$ such that
\begin{align}
  2 x b_t(x) + (p_0-1)|\sigma_t(x)|^2 \vee \int_Z|\gamma_t(x, z)|^2 \nu(dz)\leq L(M+|x|^2) \notag
\end{align}
almost surely for any $t \in [t_0, t_1]$ and $x\in \mathbb{R}^d$.
\end{assumptionA}
%---------------------------------------------------
\begin{assumptionA}   \label{as:sderc:gamma:growth}
There exist a constant $L>0$ and an $\mathscr{F}_{t_0}$-measurable random variable $N \in L(\Omega)$ such that
$$
\int_Z |\gamma_t(x,z)|^{p_0} \nu(dz) \leq L(N+|x|^{p_0})
$$
almost surely for $t \in [t_0,t_1]$  and $x \in \mathbb{R}^d$.
\end{assumptionA}
%-----------------------------------------------------
\begin{assumptionA} \label{as:convergence:localmonotonicity}
For every $R>0$,
\begin{align*}
2(x-\bar{x})&(b_t(x)-b_t(\bar{x}))+|\sigma_t(x)-\sigma_t(\bar{x})|^2
\\
& +\int_Z |\gamma_t(x,z)-\gamma_t(\bar{x},z)|^2 \nu(dz) \leq C(R)|x-\bar{x}|^2
\end{align*}
almost surely whenever $|x|\vee |\bar{x}| \leq R$ for any  $t \in [t_0, t_1]$.
\end{assumptionA}
%-----------------------------------------------------------
\begin{assumptionA} \label{as:convergence:localbound}
For every $R>0$,
$$
\sup_{|x| \leq R}|b_t(x)| \leq C(R)
$$
almost surely for any $t \in [t_0, t_1]$.
\end{assumptionA}
%---------------------------------------------------------
\begin{assumptionA} \label{as:sderc:continuity}
The function $b_t(x)$ is continuous in $x \in \mathbb{R}^d$ for every $\omega \in \Omega$ and $t \in [t_0, t_1]$.
\end{assumptionA}
%------------------------------------------------------
We make the following assumptions for the Euler-type scheme \eqref{eq:esrc}.
%------------------------------------------------------
\begin{assumptionB} \label{as:esrc:initial}
$\{x_{t_0}^n\}_{n \in \mathbb{N}}\in l_{\infty}(L^{p_0}(\Omega))$.
\end{assumptionB}
%----------------------------------------------------------
\begin{assumptionB} \label{as:esrc:coercivity}
There exist a constant $L>0$ and a sequence of $\mathscr{F}_{t_0}$-measurable random variables $\{{M}^n\}_{n \in \mathbb{N}} \in l_\infty({L}^{\frac{p_0}{2}}(\Omega))$ such that,
$$
2xb_t^n(x)+(p_0-1)|\sigma_t^n(x)|^2 \vee \int_Z |\gamma_t^n(x,z)|^2 \nu(dz) \leq L({M}^n+|x|^2)
$$
almost surely for any $t \in [t_0, t_1]$, $n \in \mathbb{N}$ and $x \in \mathbb{R}^d$.
\end{assumptionB}
%----------------------------------------------------------
\begin{assumptionB} \label{as:esrc:gamma:growth}
There exist a constant $L>0$ and a sequence of $\mathscr{F}_{t_0}$-measurable random variables $\{N^n\}_{n \in \mathbb{N}} \in l_\infty(L(\Omega))$ such that,
$$
\int_Z |\gamma_t^n(x,z)|^{p_0} \nu(dz) \leq L({N}^n+|x|^{p_0})
$$
almost surely for any $t \in [t_0, t_1]$,  $n \in \mathbb{N}$ and $x \in \mathbb{R}^d$.
\end{assumptionB}
%----------------------------------------------------------
\begin{assumptionB} \label{as:esrc:taming}
There exist a constant $L>0$ and a sequence of $\mathscr{F}_{t_0}$-measurable random variables $\{{M}^n\}_{n \in \mathbb{N}} \in l_\infty({L}^{\frac{p_0}{2}}(\Omega))$ such that,
\begin{align*}
|b_t^n(x)|^2 & \leq L n^{1/2} ({M}^n+|x|^2)
\\
|\sigma_t^n(x)|^2 & \leq L n^{1/2} ({M}^n+|x|^2)
\end{align*}
almost surely for any $t \in [t_0, t_1]$,  $n \in \mathbb{N}$ and $x \in \mathbb{R}^d$.
\end{assumptionB}
%-----------------------------------------------------------
%The assumptions mentioned below define closeness of the coefficients of SDE \eqref{eq:sderc} and scheme \eqref{eq:esrc}.
%-----------------------------------------------------------
\begin{assumptionAB}\label{as:convergence:con}
For every $R>0$,
\begin{align*}
\lim_{n \to \infty} E\int_{t_0}^{t_1}&I_{\{ {C}(R) \leq f(R)\}}\sup_{|x| \leq R} \{ |b_t(x)-b_t^n(x)|^2+|\sigma_t(x)-\sigma_t^n(x)|^2
\\
&+\int_Z |\gamma_t(x,z)-\gamma_t^n(x,z)|^2 \nu(dz) \}dt=0.
\end{align*}
\end{assumptionAB}
%-----------------------------------------------------------
\begin{assumptionAB} \label{as:convergence:initial}
The sequence $\{x_{t_0}^n\}_{n \in \mathbb{N}}$  converges in probability to $x_{t_0}$.
\end{assumptionAB}
%-----------------------------------------------------------
\begin{thm} \label{thm:convergence:rc}
Let Assumptions A-\ref{as:sderc:initial:value} to A-\ref{as:sderc:continuity}, B-\ref{as:esrc:initial} to B-\ref{as:esrc:taming},  AB-\ref{as:convergence:con} and AB-\ref{as:convergence:initial} be satisfied. Then, the explicit Euler-type scheme \eqref{eq:esrc} converges to the true solution of SDE \eqref{eq:sderc} in $\mathcal{L}^p$-sense, i.e.
$$
\lim_{n \to \infty} \sup_{t_0 \leq t \leq t_1}E|x_t-x_t^n|^{p}=0
$$
for any $p < p_0$.
\end{thm}
The proof of the above theorem can be found in Section \ref{sec:rc:proof:rate}.

For the rate of convergence of the scheme \eqref{eq:esrc}, we fix a constant $p_1\geq 2$ and consider any $p$ satisfying $p<p_1$ and $\chi p (p+\delta)/\delta \leq p_0$ for a $\delta >0$ (however small). Moreover, one replaces Assumptions A-\ref{as:convergence:localmonotonicity},  AB-\ref{as:convergence:con}  and AB-\ref{as:convergence:initial} by the following assumptions.
%-----------------------------------------------------
\begin{assumptionA} \label{as:rate:rc:monotonicity}
There exists a constant $C>0$ such that
\begin{align*}
2(x-\bar{x})&(b_t(x)-b_t(\bar{x}))+(p_1-1)|\sigma_t(x)-\sigma_t(\bar{x})|^2
\\
& \vee \int_Z|\gamma_t(x,z)-\gamma_t(\bar{x},z)|^2\nu(dz) \leq C|x-\bar{x}|^2
\end{align*}
almost surely for any $t \in [t_0,t_1]$ and $x,\bar{x} \in \mathbb{R}^d$.
\end{assumptionA}
%---------------------------------------------------------
\begin{assumptionA} \label{as:rate:rc:gamma:lipschitz}
There exist a constant $C>0$ such that
\begin{align*}
\int_Z|\gamma_t(x,z)-\gamma_t(\bar{x},z)|^{p}\nu(dz) \leq C|x-\bar{x}|^{p}
\end{align*}
almost surely for any $t \in [t_0, t_1]$ and $x,\bar{x} \in \mathbb{R}^d$
\end{assumptionA}
%--------------------------------------------------------
\begin{assumptionA} \label{as:rate:rc:polylipschitz}
There exist a constant  $C>0$ and $\chi>0$ such that
\begin{align*}
|b_t(x)-b_t(\bar{x})| &\leq C(1+|x|^\chi+|\bar{x}|^\chi)|x-\bar{x}|
\end{align*}
almost surely for any $t \in [t_0, t_1]$ and $x,\bar{x} \in \mathbb{R}^d$.
\end{assumptionA}
%--------------------------------------------------------
\begin{assumptionB} \label{as:rate:rc:bn:polygrowth}
There exist constants $L>0$, $\chi>0$ and a sequence of $\mathscr{F}_{t_0}$-measurable random variables $\{{M}^n\}_{n \in \mathbb{N}} \in l_\infty({L}^{p_0}(\Omega))$ such that,
\begin{align*}
|b_t^n(x)| \leq L({M}^n+|x|^{\chi+1})
\end{align*}
almost surely for any $t \in [t_0, t_1]$, $n \in \mathbb{N}$ and $x \in \mathbb{R}^d$.
\end{assumptionB}%--------------------------------------------------------
\begin{assumptionAB} \label{as:rate:rc:rate}
There exists a constant $L>0$ such that, for every $n \in \mathbb{N}$,
\begin{align*}
E\int_{t_0}^{t_1} &\big\{|b_t(x_{\kappa(n,t)}^n)-b_t^n(x_{\kappa(n,t)}^n)|^p+|\sigma_t(x_{\kappa(n,t)}^n)-\sigma_t^n(x_{\kappa(n,t)}^n)|^p \notag
\\
&+\Big(\int_Z|\gamma_t(x_{\kappa(n,t)}^n,z)-\gamma_t^n(x_{\kappa(n,t)}^n,z)|^\rho \nu(dz) \Big)^\frac{p}{\rho} \big\} dt \leq L n^{-\frac{p}{p+\delta}}
\end{align*}
for $\rho=2, p$.
\end{assumptionAB}
%--------------------------------------------------------
\begin{assumptionAB} \label{as:rate:rc:initial}
There exists a constant $L>0$ such that,
\begin{align*}
E|x_{t_0}-x_{t_0}^n|^p \leq Ln^{-\frac{p}{p+\delta}}
\end{align*}
for every $n \in \mathbb{N}$.
\end{assumptionAB}
%----------------------------------------------------------
\begin{thm} \label{thm:rate:rc}
Let Assumptions A-\ref{as:sderc:initial:value} to A-\ref{as:sderc:gamma:growth}, A-\ref{as:convergence:localbound}, A-\ref{as:rate:rc:monotonicity} to A-\ref{as:rate:rc:polylipschitz}, B-\ref{as:esrc:initial} to B-\ref{as:rate:rc:bn:polygrowth},  AB-\ref{as:rate:rc:rate} and AB-\ref{as:rate:rc:initial} hold. Then, the explicit Euler-type scheme \eqref{eq:esrc} converges to the true solution of SDE \eqref{eq:sderc} in  $\mathcal{L}^p$-sense with a rate of convergence arbitrarily close to $1/p$ i.e., for every $n \in \mathbb{N}$,
\begin{align*}
\sup_{t_0 \leq t \leq t_1}E|x_t-x_t^n|^{p} \leq Kn^{-\frac{p}{p+\delta}}
\end{align*}
where the positive constant $K$ does not depend on $n$.
\end{thm}
The proof of the above theorem can be found in Section \ref{sec:rc:proof:rate}. Notice that the optimal rate of convergence in the above theorem is attained for $p=2$ which is arbitrarily close to $0.5$. Moreover, the rate of convergence coincide with that of the classical Euler scheme. In the following two sections, we provide examples of SDE and SDDE that can fit into our model.
%----------------------------------------------------------
\subsection{Explicit Euler-type scheme for SDE driven by L\'evy noise}
Let $\beta_t(x)$ and $\alpha_t(x)$ be $\mathscr{B}([0,T])\otimes \mathscr{B}(\mathbb{R}^d)$-measurable functions in $\mathbb{R}^d$ and $\mathbb{R}^{d \times m}$ respectively. Also, $\lambda_t(x,z)$ is a $\mathscr{B}([0,T])\otimes \mathscr{B}(\mathbb{R}^d)\otimes \mathscr{Z}$-measurable function in $\mathbb{R}^d$. We consider the following SDE,
\begin{align} \label{eq:sde}
dx_t=\beta_t(x_t)dt+\alpha_t(x_t)dw_t+\int_{Z} \lambda_t(x_t,z) \tilde{N}(dt,dz)
\end{align}
almost surely for any $t \in [0,T]$ with initial value $x_0 \in L^{p_0}(\Omega)$. Notice that one defines SDE \eqref{eq:sde} as a special case of SDE \eqref{eq:sderc} with $t_0=0$, $t_1=T$ and
\begin{align*}
b_t(x):=\beta_t(x), \sigma_t(x):=\alpha_t(x), \gamma_t(x,z):=\lambda_t(x,z)
\end{align*}
for any $t \in [0,T]$ and $x \in \mathbb{R}^d$. Moreover, in the assumptions listed above on the coefficients of SDE \eqref{eq:sderc}, one uses ${M}\equiv 1$ whereas for every $R>0$, ${C}(R)$ is a positive constant.  Similarly, one can define an explicit Euler-type scheme of SDE \eqref{eq:sde} as a special case of the scheme \eqref{eq:esrc} with the following mappings,
\begin{align*}
b_t^n(x):=\frac{\beta_t(x)}{1+n^{-1/2}|x|^{2\chi}}, \sigma_t^n(x):=\frac{\alpha_t(x)}{1+n^{-1/2}|x|^{2\chi}}, \gamma_t^n(x,z):= \lambda_t(x,z)
\end{align*}
for any $n \in \mathbb{N}$, $t \in [0,T]$, $x \in \mathbb{R}^d$ and $z \in Z$ with $x_0^n=x_0$.  It is easy to verify that Assumptions B-\ref{as:esrc:coercivity} to B-\ref{as:rate:rc:bn:polygrowth}, AB-\ref{as:convergence:con} and AB-\ref{as:rate:rc:rate} are satisfied. Hence, the results of Theorems [\ref{thm:convergence:rc}, \ref{thm:rate:rc}] hold true.
\begin{rem} \label{rem:about:improv}
Notice that in the above example, coefficients of the SDE \eqref{eq:sderc} and the scheme \eqref{eq:esrc} are deterministic. In this case, one can use the following condition on $b_t^n(x)$ in Assumption B-\ref{as:esrc:taming},
\begin{align*}
|b_t^n(x)| \leq L n^{1/2} (1+|x|)
\end{align*}
with the below mentioned coefficients,
\begin{align*}
b_t^n(x)=\frac{\beta_t(x)}{1+n^{-1/2}|x|^\chi}
\end{align*}
for any $t \in [t_0, t_1]$, $n \in \mathbb{N}$ and $x \in \mathbb{R}^d$. The proof of the Lemma \ref{lem:esrc:momentbound} is then followed in similar way as done in \cite{kumar2016b, sabanis2015} because in such a case, $b_t(x_{\kappa(n,t)}^n)$ remains $\mathscr{F}_{\kappa(n,t)}$-measurable in order to eliminate the stochastic integral in the second term of the right hand side of \eqref{eq:esrc:for:remark}.  Hence, this approach does not increase the moment bound requirements on the initial value as has been attained in \cite{sabanis2015}.
\end{rem}
%----------------------------------------------------------
\subsection{Explicit Euler-type scheme for SDDE driven by L\'evy noise} \label{SDDE section}
Let $\beta_t(y_1,\ldots, y_k,x)$ and $\alpha_t(y_1,\ldots, y_k,x)$ be $\mathscr{B}([0,T])\otimes \mathscr{B}(\mathbb{R}^{d\times k})\otimes \mathscr{B}(\mathbb{R}^{d})$-measurable functions in $\mathbb{R}^d$ and $\mathbb{R}^{d \times m}$ respectively. Also, $\lambda_t(y_1,\ldots, y_k,x)$ is a $\mathscr{B}([0,T])\otimes \mathscr{B}(\mathbb{R}^{d\times k})\otimes \mathscr{B}(\mathbb{R}^{d}) \otimes \mathscr{Z}$-measurable function in $\mathbb{R}^d$. Further, let $d_1(t),\ldots,d_k(t)$ be increasing functions of $t$ satisfying $-H \leq d_i(t) \leq \lfloor t/h \rfloor h$ for fixed constants $h>0$ and $H>0$ for every $i=1,\ldots,k$. We consider the following SDDE,
\begin{align} \label{eq:sdde}
dx_t=\beta_t(y_t,x_t)dt+\alpha_t(y_t,x_t)dw_t+\int_Z \lambda_t(y_t,x_t,z) \tilde{N}(dt,dz)
\end{align}
almost surely for any $t \in [0,T]$ with initial data $x_t=\xi_t$ for any $t \in [-H,0]$ satisfying $E\sup_{-H \leq t \leq 0}|\xi_t|^{p_0}<\infty$, where $y_t:=(x_{d_1(t)}, \ldots, x_{d_k(t)})$. The SDDE \eqref{eq:sdde} can be regarded as a special case of SDE \eqref{eq:sderc} with the following mappings,
\begin{align*}
b_t(x):=\beta_t(y_t,x), \sigma_t(x):=\alpha_t(y_t,x), \gamma_t(x_t,z):=\lambda_t(y_t,x,z)
\end{align*}
almost surely for any $t\in [0,T]$, $x \in \mathbb{R}^d$ and $z \in Z$. Suppose that the function $\beta_t(y,x)$ satisfies $\beta_t(y,x) \leq L(1+|y|^{\chi_1}+|x|^{\chi_2})$ for any $y \in \mathbb{R}^{d \times k}$ and $x \in \mathbb{R}^d$, where $L$, $\chi_1$ and $\chi_2$ are  positive constants.  Then, the explicit Euler-type scheme of SDDE \eqref{eq:sdde} can be defined with the following mappings,
\begin{align*}
b_t^n(x):=\frac{\beta_t(y_t^n,x)}{1+n^{-1/2}(|y_t^n|^{2 \chi_1}+|x|^{2\chi_2})},& \sigma^n_t(x):=\frac{\alpha_t(y_t^n,x)}{1+n^{-1/2}(|y_t^n|^{2 \chi_1}+|x|^{2\chi_2})}
\\
\gamma_t^n(x, z):=&\lambda_t(y_t^n,x, z)
\end{align*}
almost surely for any $t \in [0,T]$ and $x \in \mathbb{R}^d$. By adopting the approach of \cite{dare2016}, one can show that Theorems [\ref{thm:convergence:rc}, \ref{thm:rate:rc}] hold true.
%-----------------------------------------------------------
\section{Moment Bounds}
%-----------------------------------------------------------
We make the following observations.
\begin{rem} \label{rem:rc:local:bound:sig:gam}
Due to Assumptions A-\ref{as:sderc:coercivity} and A-\ref{as:convergence:localbound}, for every $R>0$,
\begin{align*}
|\sigma_t(x)|^2 \vee \int_Z|\gamma_t(x,z)|^2 \nu(dz) \leq {C}(R)
\end{align*}
almost surely for any $t \in [t_0, t_1]$ whenever $|x| \leq R$.
\end{rem}
%--------------------------------------------------------
%The following lemma is used in this article which can be found in \cite{mikulevicius2012}.
%\begin{lem}
%Suppose $r \geq 2$ and $T>0$. There exists a constant $K$ which depends only on $r$ such that for every $\mathscr{P}\otimes \mathscr{Z}$-measurable real-valued  function $g$ that satisfy
%$$
%\int_0^T \int_Z |g_t(z)|^2\nu(dz)dt < \infty
%$$
%almost surely, the following holds,
%\begin{align*}
%E\sup_{0 \leq t \leq T}\Big|\int_0^t\int_Z g_s(z)&\tilde{N}(ds,dz)\Big|^r \leq  K E\Big(\int_0^T\int_Z |g_t(z)|^2 \nu(dz)dt\Big)^\frac{r}{2}
%\\
%& +K E\int_0^T\int_Z |g_t(z)|^r \nu(dz)dt.
%\end{align*}
%If $1\leq r \leq 2$, then the second term on the right hand side can be dropped.
%\end{lem}
%-----------------------------------------------------------
%-----------------------------------------------------------
The moment bound of SDE \eqref{eq:sderc} is well know, but for the completeness of the article, we prove this in the following lemma.
\begin{lem} \label{lem:mb:rc}
%---------------------------------------------------
Let Assumptions A-\ref{as:sderc:initial:value} to A-\ref{as:sderc:continuity}  be satisfied, then there exists a unique solution $\{x_t\}_{t\in [t_0,t_1]}$ of SDE \eqref{eq:sderc}. Moreover,
$$
\sup_{t_0 \leq t \leq t_1}E|x_t|^{p_0} \leq K,
$$
where $K$ is a positive constant.
\end{lem}
%---------------------------------------------
\begin{proof}
%---------------------------------------------
The proof of existence and uniqueness of the solution of SDE \eqref{eq:sderc} can be found in \cite{gyongy1980} under more general settings than those considered here.

Define a stopping time $\tilde{\tau}_R:= \inf\{t \geq t_0: |x_t| >R\}\wedge t_1$ and notice that $|x_{t-}| \leq R$ for any $ t_0 \leq t \leq \tilde{\tau}_R$.  By using  It\^o's formula,
\begin{align} \label{eq:sderc:ito}
|x_t|^{p_0} &= |x_{t_0}|^{p_0}+ {p_0} \int_{t_0}^{t} |x_s|^{p_0-2} x_s b_s( x_s) ds + p_0\int_{t_0}^{t} |x_s|^{p_0-2} x_s \sigma_s( x_s)  dw_s  \notag
\\
& + \frac{p_0(p_0-2)}{2} \int_{t_0}^{t} |x_s|^{p_0-4}|\sigma_s^{*}(x_s) x_s|^2ds +\frac{p_0}{2}\int_{t_0}^{t} |x_s|^{p_0-2}|\sigma_s(x_s)|^2 ds  \notag
\\
&+  p_0\int_{t_0}^{t} \int_{Z} |x_s|^{p_0-2} x_{s} \gamma_s( x_{s},z)    \tilde N(ds,dz) \notag
\\
+\int_{t_0}^{t} &\int_{Z}\{ |x_{s}+\gamma_s( x_{s},z)|^{p_0}-|x_{s}|^{p_0}-p_0|x_{s}|^{p_0-2} x_{s}\gamma_s( x_{s},z) \}N(ds,dz)
\end{align}
almost surely for any $t \in [t_0, t_1]$. Now, on taking expectation and using Schwarz inequality, one obtains,
\begin{align} \label{eq:thirdterm}
E|x_{t \wedge \tilde{\tau}_R}|^{p_0} &\leq  E|x_{t_0}|^{p_0}+ \frac{p_0}{2} E\int_{t_0}^{t \wedge \tilde{\tau}_R} |x_s|^{p_0-2} \{2x_s b_s( x_s)  + (p_0-1)|\sigma_s(x_s)|^2\}ds \notag
\\
+E\int_{t_0}^{t \wedge \tilde{\tau}_R} \hspace{-2mm}&\int_{Z}\{ |x_{s}\hspace{-.5mm}+\hspace{-.5mm}\gamma_s( x_{s},z)|^{p_0}\hspace{-.5mm}-\hspace{-.5mm}|x_{s}|^{p_0}\hspace{-.5mm}-\hspace{-.5mm}p_0|x_{s}|^{p_0-2} x_{s}\gamma_s( x_{s},z) \}\nu(dz)ds
\end{align}
for any $t \in [t_0, t_1]$. One notes that when $p_0=2$, then
\begin{align*}
E|x_{t \wedge \tilde{\tau}_R}|^{2} \leq  E|x_{t_0}|^{2}\hspace{-.5mm}+\hspace{-.5mm}  E\int_{t_0}^{t \wedge \tilde{\tau}_R}  \{2x_s b_s( x_s) \hspace{-.5mm} + \hspace{-.5mm} |\sigma_s(x_s)|^2\hspace{-.5mm}+\hspace{-.5mm}\int_{Z}|\gamma_s( x_{s},z)|^{2}\nu(dz)\}ds \notag
\end{align*}
for any $t \in [t_0,t_1]$. Thus, the application of Assumption A-\ref{as:sderc:coercivity}, Gronwall's inequality and Fatou's lemma completes the proof for the case $p_0=2$. For the case $p_0 \geq 4$,  one uses the formula for the remainder  and obtains the following estimates,
\begin{align}
E&|x_{t \wedge \tilde{\tau}_R}|^{p_0} \leq  E|x_{t_0}|^{p_0}+ \frac{p_0}{2} E\int_{t_0}^{t \wedge \tilde{\tau}_R} |x_s|^{p_0-2} \big\{2x_s b_s( x_s)  + (p_0-1)|\sigma_s(x_s)|^2\big\}ds \notag
\\
&\qquad +KE\int_{t_0}^{t \wedge \tilde{\tau}_R} \int_{Z} |x_{s}|^{p_0-2}|\gamma_s( x_{s},z)|^2\nu(dz)ds \notag
\\
&\qquad +KE\int_{t_0}^{t \wedge \tilde{\tau}_R} \int_{Z} |\gamma_s( x_{s},z)|^{p_0}\nu(dz)ds \notag
\end{align}
for any $t \in[t_0,t_1]$. On the application of Assumptions A-\ref{as:sderc:coercivity} and A-\ref{as:sderc:gamma:growth}, one obtains,
\begin{align}
\sup_{t_0 \leq t \leq u}E|x_{t \wedge \tilde{\tau}_R}|^{p_0} &\leq  E|x_{t_0}|^{p_0}+ K+K \int_{t_0}^{u} \sup_{t_0 \leq r \leq s}E|x_{r  \wedge \tilde{\tau}_R}|^{p_0} ds  < \infty \notag
\end{align}
for any $u \in[t_0,t_1]$. Hence, the application of Gronwall's lemma and Fatou's lemma completes the proof.
\end{proof}
%---------------------------------------------------------
Before proving the moment bound of the scheme \eqref{eq:esrc}, we prove  the following lemma.
%---------------------------------------------------------
\begin{lem} \label{lem:rc:one-step}
Let Assumptions B-\ref{as:esrc:coercivity} to B-\ref{as:esrc:taming} be satisfied. Then, for every $ \rho \in (2, p_0]$, the following holds
\begin{align}
E( |x_t^n-x_{\kappa(n,t)}^n|^{\rho}|\mathscr{F}_{\kappa(n,t)}) \leq K(n^{-\frac{\rho}{4}}( |{M}^n|^\frac{\rho}{2} +|x_{\kappa(n,t)}^n|^{\rho})+n^{-1}( N^n+|x_{\kappa(n,t)}^n|^{\rho})) \notag
\end{align}
almost surely and for every $\rho \in [1, 2]$, the following holds
\begin{align*}
E\big( |x_t^n-x_{\kappa(n,t)}^n|^{\rho}|\mathscr{F}_{\kappa(n,t)}\big)  &\leq Kn^{-\frac{\rho}{4}}(|M^n|^\frac{\rho}{2}+|x_{\kappa(n,t)}^n|^{\rho})
\end{align*}
almost surely for any $t \in [t_0, t_1]$, where the positive constant $K$ does not depend on $n$.
\end{lem}
%---------------------------------------------------------
\begin{proof}
By equation \eqref{eq:esrc}, one obtains
\begin{align*}
E\big( |x_t^n&-x_{\kappa(n,t)}^n|^{\rho}|\mathscr{F}_{\kappa(n,t)}\big)  \leq K E\Big(\Big|\int_{\kappa(n,t)}^{t} b_s^n(x_{\kappa(n,s)}^n) ds\Big|^{\rho}|\mathscr{F}_{\kappa(n,t)}\Big)
\\
& +K E\Big(\Big|\int_{\kappa(n,t)}^{t} \sigma_s^n(x_{\kappa(n,s)}^n) dw_s\Big|^{\rho} |\mathscr{F}_{\kappa(n,t)}\Big)
\\
& +KE\Big(\Big|\int_{\kappa(n,t)}^{t} \int_Z \gamma_s^n(x_{\kappa(n,s)}^n, z) \tilde{N}(ds, dz)\Big|^{\rho}|\mathscr{F}_{\kappa(n,t)}\Big)
\end{align*}
which on the application of H\"older's inequality and an elementary inequality of stochastic integrals gives,
\begin{align*}
E\big( |x_t^n&-x_{\kappa(n,t)}^n|^{\rho}|\mathscr{F}_{\kappa(n,t)}\big)  \leq K n^{-\rho+1}E\Big(\Big(\int_{\kappa(n,t)}^{t} |b_s^n(x_{\kappa(n,s)}^n)|^{\rho} ds\Big)|\mathscr{F}_{\kappa(n,t)}\Big)
\\
& +K E\Big(\Big(\int_{\kappa(n,t)}^{t} |\sigma_s^n(x_{\kappa(n,s)}^n)|^{2} ds \Big)^\frac{\rho}{2}|\mathscr{F}_{\kappa(n,t)}\Big)
\\
& +KE\Big(\Big(\int_{\kappa(n,t)}^{t} \int_Z |\gamma_s^n(x_{\kappa(n,s)}^n, z)|^2 \nu(dz)ds \Big)^\frac{\rho}{2}|\mathscr{F}_{\kappa(n,t)}\Big)
\\
& +KE\Big(\int_{\kappa(n,t)}^{t} \int_Z |\gamma_s^n(x_{\kappa(n,s)}^n, z)|^\rho \nu(dz)ds|\mathscr{F}_{\kappa(n,t)}\Big)
\end{align*}
for any $t \in [t_0, t_1]$. Notice that when $\rho \in [1,2]$, then the last term on the right hand side of the above inequality can be dropped. Furthermore, one uses Assumptions B-\ref{as:esrc:coercivity}, B-\ref{as:esrc:gamma:growth} and B-\ref{as:esrc:taming} to complete the proof.
\end{proof}
%---------------------------------------------------------
\begin{lem} \label{lem:esrc:momentbound}
Let Assumptions B-\ref{as:esrc:initial} to B-\ref{as:esrc:taming} be satisfied, then the following holds
$$
\sup_{n \in \mathbb{N}}\sup_{t_0 \leq t \leq t_1}E|x_t^n|^{p_0} \leq K,
$$
where $K$ is a positive constant and does not depend on $n$.
\end{lem}
%--------------------------------------------------------
\begin{proof}
For every $n \in \mathbb{N}$, one applies the It\^{o}'s formula to obtain,
\begin{align} \label{eq:esrc:ito}
|x_t^n|^{p_0} &= |x_{t_0}^n|^{p_0}+{p_0} \int_{t_0}^{t} |x_s^n|^{p_0-2} x_s^n b_s^n( x_{\kappa(n,s)}^n) ds \notag
\\
&+ p_0\int_{t_0}^{t} |x_s^n|^{p_0-2} x_s^n \sigma_s^n( x_{\kappa(n,s)}^n)dw_s  \notag
\\
& + \frac{p_0(p_0-2)}{2} \int_{t_0}^{t} |x_s^n|^{p_0-4}|\sigma_s^{n*}( x_{\kappa(n,s)}^n) x_s^n|^2ds \notag
\\
&+\frac{p_0}{2}\int_{t_0}^{t} |x_s^n|^{p_0-2}|\sigma_s^n( x_{\kappa(n,s)}^n)|^2 ds  \notag
\\
&+  p_0\int_{t_0}^{t} \int_{Z} |x_s^n|^{p_0-2} x_{s}^n \gamma_s^n( x_{\kappa(n,s)}^n,z)    \tilde N(ds,dz) \notag
\\
+\int_{t_0}^{t} &\int_{Z} \hspace{-1mm}\{ |x_{s}^n+\gamma_s^n( x_{\kappa(n,s)}^n,z)|^{p_0} \hspace{-1mm}-\hspace{-1mm}|x_{s}^n|^{p_0}\hspace{-1mm}-\hspace{-1mm}p_0|x_{s}^n|^{p_0-2} x_{s}^n\gamma_s^n(  x_{\kappa(n,s)}^n,z) \}N(ds,dz)
\end{align}
almost surely for any $t \in [t_0, t_1]$. The last term on the right hand side of the above equation can be estimated by the formula for the remainder as before. Hence, on taking expectation and using Schwarz inequality, one obtains the following estimates,
\begin{align}
E|x_t^n|^{p_0} &\leq E|x_{t_0}^n|^{p_0}+p_0 E\int_{t_0}^{t} |x_s^n|^{p_0-2} (x_s^n-x_{\kappa(n,s)}^n) b_s^n( x_{\kappa(n,s)}^n) ds   \notag
\\
&+\frac{p_0}{2} E\int_{t_0}^{t} |x_s^n|^{p_0-2} \big\{2x_{\kappa(n,s)}^n b_s^n( x_{\kappa(n,s)}^n) +(p_0-1) |\sigma_s^{n}( x_{\kappa(n,s)}^n)|^2\}ds \notag
\\
&+K E\int_{t_0}^{t} \int_{Z} |x_{s}^n|^{p_0-2}|\gamma_s^n( x_{\kappa(n,s)}^n,z)|^2 \nu(dz)ds \notag
\\
&+K E\int_{t_0}^{t} \int_{Z} |\gamma_s^n( x_{\kappa(n,s)}^n,z)|^{p_0} \nu(dz)ds \label{eq:esrc:for:remark}
\end{align}
which due to Schwarz inequality, Assumptions B-\ref{as:esrc:coercivity}, B-\ref{as:esrc:gamma:growth} and B-\ref{as:esrc:taming} yields,
\begin{align}
E|x_t^n|^{p_0} &\leq E|x_{t_0}^n|^{p_0}+K n^\frac{1}{4} E\int_{t_0}^{t} |x_s^n|^{p_0-2} |x_s^n-x_{\kappa(n,s)}^n| (|{M}^n|^\frac{1}{2}+| x_{\kappa(n,s)}^n|) ds   \notag
\\
&+K E\int_{t_0}^{t} |x_s^n|^{p_0-2} ({M}^n+|x_{\kappa(n,s)}^n|^2 )ds+K E\int_{t_0}^{t}  ({N}^n+| x_{\kappa(n,s)}^n|^{p_0})ds \notag
\end{align}
for any $t \in [t_0, t_1]$. Moreover, one uses Young's inequality and an algebraic inequality to obtain the following estimates,
\begin{align}
E|x_t^n|^{p_0} &\leq E|x_{t_0}^n|^{p_0}+K n^\frac{1}{4} E\int_{t_0}^{t}  |x_s^n-x_{\kappa(n,s)}^n|^{p_0-1} (|{M}^n|^\frac{1}{2}+| x_{\kappa(n,s)}^n|) ds   \notag
\\
&+K n^\frac{1}{4} E\int_{t_0}^{t}  |x_{\kappa(n,s)}^n|^{p_0-2}|x_s^n-x_{\kappa(n,s)}^n|(|{M}^n|^\frac{1}{2}+| x_{\kappa(n,s)}^n|) ds   \notag
\\
&+K+ K \int_{t_0}^{t} E|x_s^n|^{p_0} ds + K \int_{t_0}^{t} E|x_{\kappa(n,s)}^n|^{p_0} ds  \notag
\end{align}
for any $t \in [t_0, t_1]$. Also, one notices that for $p_0=2$, the second and third terms on the right hand side of the above inequality  are same which can be kept in mind in the following calculations. Moreover, the above can also be written as,
\begin{align}
E|x_t^n&|^{p_0} \leq E|x_{t_0}^n|^{p_0}\hspace{-1mm}+K n^\frac{1}{4} E\hspace{-1mm}\int_{t_0}^{t}  (|M^n|^\frac{1}{2}\hspace{-1mm}+| x_{\kappa(n,s)}^n|)E(|x_s^n\hspace{-1mm}-\hspace{-1mm}x_{\kappa(n,s)}^n|^{p_0-1}|\mathscr{F}_{\kappa(n,s)})  ds   \notag
\\
&+K n^\frac{1}{4} E\int_{t_0}^{t}  |x_{\kappa(n,s)}^n|^{p_0-2}(|M^n|^\frac{1}{2}+| x_{\kappa(n,s)}^n|) E(|x_s^n-x_{\kappa(n,s)}^n||\mathscr{F}_{\kappa(n,s)}) ds   \notag
\\
&+K+ K \int_{t_0}^{t} \sup_{t_0 \leq r \leq s} E|x_r^n|^{p_0} ds   \notag
\end{align}
for any $t \in [t_0, t_1]$. Notice that when $p_0 \in [2,3]$, then one uses the case $\rho \in [1,2]$ in Lemma \ref{lem:rc:one-step} which gives the rate $n^{-\rho/4}$ and hence $n^{1/4}$ disappears from the second and third terms. When $p_0 \geq 3$, then the rate is $n^{-1}$ which cancels out $n^{1/4}$ in the second term. As a consequence, one obtains
\begin{align*}
\sup_{t_0 \leq t \leq u}E|x_t^n|^{p_0} \leq K+K\int_{t_0}^{u}\sup_{t_0 \leq t \leq s}E|x_r^n|^{p_0} ds < \infty
\end{align*}
for $u \in [t_0, t_1]$ where $K$ does not depend on $n$. The finiteness of the right hand side of the above inequality is guaranteed as one can easily show by adapting similar arguments as those in Lemma \ref{lem:mb:rc} that,
\begin{align*}
\sup_{t_0 \leq t \leq t_1}E|x_t^n|^{p_0} \leq  \tilde{K}
\end{align*}
where a priori it is not clear whether the constant $\tilde{K}$ is independent of $n$ or not. The application of Gronwall's lemma completes the proof.
\end{proof}
\section{Proof of Main Results} \label{sec:rc:proof:rate}
First, we make the following observations.
%-----------------------------------------------------------
\begin{rem} \label{rem:convergence:siggam:con}
Due to Assumptions A-\ref{as:convergence:localmonotonicity} and A-\ref{as:convergence:localbound}, for every $R>0$,
\begin{align*}
|\sigma_t(x)-\sigma_t(\bar{x})|^2+\int_Z |\gamma_t(x,z)-\gamma_t(\bar{x},z)|^2 \nu(dz) \leq {C}(R)(|x-\bar{x}|^2+|x-\bar{x}|)
\end{align*}
almost surely whenever $|x|\vee |\bar{x}| \leq R$ for any $t \in [t_0, t_1]$ and $x,\bar{x} \in \mathbb{R}^d$.
\end{rem}
%-----------------------------------------------------------
For proving Theorem \ref{thm:convergence:rc}, one requires the following result.
\begin{cor} \label{cor:convergence:onestep}
Let Assumptions B-\ref{as:esrc:initial} to B-\ref{as:esrc:taming} be satisfied. Then for any $\rho \in [2,p_0]$, the following holds,
\begin{align*}
\sup_{t_0 \leq t \leq t_1}E|x_t^n-x_{\kappa(n,t)}^n|^\rho \leq K(n^{-\frac{\rho}{4}}+n^{-1})
\end{align*}
and for any $\rho \in [1,2]$, the following holds,
\begin{align*}
\sup_{t_0 \leq t \leq t_1}E|x_t^n-x_{\kappa(n,t)}^n|^\rho \leq Kn^{-\frac{\rho}{4}}
\end{align*}
where $K$ is a positive constant that does not depend on $n$.
\end{cor}
\begin{proof}
The proof follows immediately due to Lemmas [\ref{lem:rc:one-step}, \ref{lem:esrc:momentbound}].
\end{proof}
%---------------------------------------------------------
\begin{proof}[\textit{\textbf{Proof of Theorem \ref{thm:convergence:rc}}}]
For every $n \in \mathbb{N}$ and $R>0$, define the following the stopping times,
\begin{align*}
\tilde{\tau}_R:=\inf\{t \geq t_0: |x_t| \geq &R\}, \bar{\tau}_{nR}:=\inf\{t \geq t_0: |x_t^n| \geq R\}
\\
&\tau_{nR}:= \tilde{\tau}_R \wedge \bar{\tau}_{nR}
\end{align*}
almost surely. Then, one can write,
\begin{align}
\sup_{t_0 \leq t \leq t_1}E|x_t-x_t^n|^2 & \leq \sup_{t_0 \leq t \leq t_1} E|x_t-x_t^n|^2 I_{\{\{\tilde{\tau}_R \leq t_1\}\cup \{\bar{\tau}_{nR} \leq t_1\} \cup\{{C}(R) >f(R)\}\}} \notag
\\
& + \sup_{t_0 \leq t \leq t_1}E|x_t-x_t^n|^2 I_{\{\{\tilde{\tau}_R > t_1\}\cap \{\bar{\tau}_{nR} > t_1\} \cap\{{C}(R) \leq f(R)\}\}} \notag
\\
=:T_1&+T_2. \label{eq:convergence:T1+T2}
\end{align}
For $T_1$, one uses H\"older's inequality and obtains the following,
\begin{align*}
T_1&:=\sup_{t_0 \leq t \leq t_1}E|x_t-x_t^n|^2 I_{\{\tilde{\tau}_R \leq t_1, \bar{\tau}_{nR} \leq t_1, {C}(R) >f(R)\}}
\\
& \leq \Big(\sup_{t_0 \leq t \leq t_1}E|x_t-x_t^n|^{p_0}\Big)^\frac{2}{p_0}\{P(\tilde{\tau}_R \leq t_1, \bar{\tau}_{nR} \leq t_1, {C}(R) >f(R))\}^\frac{p_0-2}{p_0}
\end{align*}
which on the application of Lemmas [\ref{lem:mb:rc},  \ref{lem:esrc:momentbound}] yields,
\begin{align}
T_1 & \leq K\Big(\frac{E|x_{\tilde{\tau}_R}|^{p_0}}{R^{p_0}} +\frac{E|x_{\bar{\tau}_{nR}}^n|^{p_0}}{R^{p_0}}+P(C(R)>f(R))\Big)^\frac{p_0-2}{p_0} \notag
\\
& \leq K\Big(\frac{1}{R^{p_0}}+P({C}(R)>f(R))\Big)^\frac{p_0-2}{p_0} \label{eq:convergence:T1}
\end{align}
for  every $R>0$.

Moreover, one notices that $T_2$ can be estimated by,
\begin{align}
T_2&:=\sup_{t_0 \leq t \leq t_1}E|x_t-x_t^n|^{2}I_{\{\{\tilde{\tau}_R>t_1\}\cap \{ \bar{\tau}_{nR}>t_1\} \cap \{{C}(R) \leq f(R)\}\}} \notag
\\
& \leq  \sup_{t_0 \leq t \leq t_1}E|x_{t\wedge \tau_{nR}}-x_{t\wedge \tau_{nR}}^n|^{2}I_{\{{C}(R) \leq f(R)\}}. \label{eq:convergence:T2}
\end{align}
Also, due to equations \eqref{eq:sderc} and \eqref{eq:esrc},
\begin{align}
x_t-x_t^n&=x_{t_0}-x_{t_0}^n+\int_{t_0}^t \{b_s(x_s)-b_s^n(x_{\kappa(n,s)}^n)\} ds \notag
\\
&+\int_{t_0}^t \{\sigma_s(x_s)-\sigma_s^n(x_{\kappa(n,s)}^n)\} dw_s \notag
\\
&+\int_{t_0}^t \int_Z\{\gamma_s(x_s,z)-\gamma_t^n(x_{\kappa(n,s)}^n,z)\} \tilde{N}(ds, dz) \label{eq:rc:x-xn}
\end{align}
for any $t \in [t_0, t_1]$. Now, one  uses It\^o's formula to obtain the following,
\begin{align}
|x_t-x_t^n|^{2} &= |x_{t_0}-x_{t_0}^n|^{2}+{2} \int_{t_0}^{t}  (x_s-x_s^n) (b_s(x_s)-b_s^n( x_{\kappa(n,s)}^n) )ds \notag
\\
&+ 2\int_{t_0}^{t}  (x_s-x_s^n) \{\sigma_t(x_s)-\sigma_s^n( x_{\kappa(n,s)}^n)\}dw_s \notag
\\
& +\int_{t_0}^{t} |\sigma_t(x_s)-\sigma_s^n( x_{\kappa(n,s)}^n)|^2 ds  \notag
\\
&+  2\int_{t_0}^{t} \int_{Z}  (x_t-x_t^n) \{\gamma_t(x_t,z)-\gamma_s^n( x_{\kappa(n,s)}^n,z)\}    \tilde N(ds,dz) \notag
\\
&+\int_{t_0}^{t} \int_{Z} |\gamma_s(x_s,z)-\gamma_s^n( x_{\kappa(n,s)}^n,z)|^{2} N(ds,dz) \notag
\end{align}
almost surely for any $t \in [t_0, t_1]$. By taking expectation one gets,
\begin{align}
E|x_{t\wedge \tau_{nR}}&-x_{t\wedge \tau_{nR}}^n|^{2}I_{\{C(R) \leq f(R)\}} = E|x_{t_0}-x_{t_0}^n|^{2}I_{\{C(R) \leq f(R)\}} \notag
\\
&+{2} E\int_{t_0}^{t\wedge \tau_{nR}} I_{\{C(R) \leq f(R)\}} (x_s-x_s^n) (b_s(x_s)-b_s^n( x_{\kappa(n,s)}^n) )ds \notag
\\
&+E\int_{t_0}^{t\wedge \tau_{nR}} I_{\{C(R) \leq f(R)\}}|\sigma_s(x_s)-\sigma_s^n( x_{\kappa(n,s)}^n)|^2 ds  \notag
\\
&+E\int_{t_0}^{t\wedge \tau_{nR}} I_{\{C(R) \leq f(R)\}}\int_{Z} |\gamma_s(x_s,z)-\gamma_s^n( x_{\kappa(n,s)}^n,z)|^{2} \nu(dz) ds \notag
\end{align}
which further implies,
\begin{align}
E|x_{t\wedge \tau_{nR}}&-x_{t\wedge \tau_{nR}}^n|^{2}I_{\{C(R) \leq f(R)\}} = E|x_{t_0}-x_{t_0}^n|^{2}I_{\{C(R) \leq f(R)\}} \notag
\\
&+2 E\int_{t_0}^{t\wedge \tau_{nR}}  I_{\{C(R) \leq f(R)\}}(x_s-x_{\kappa(n,s)}^n)(b_s(x_s)-b_s( x_{\kappa(n,s)}^n) ) ds \notag
\\
&+2 E\int_{t_0}^{t\wedge \tau_{nR}}I_{\{C(R) \leq f(R)\}}  (x_s-x_{\kappa(n,s)}^n)(b_s( x_{\kappa(n,s)}^n)-b_s^n( x_{\kappa(n,s)}^n) ) ds \notag
\\
&+ 2 E\int_{t_0}^{t\wedge \tau_{nR}} I_{\{C(R) \leq f(R)\}}(x_{\kappa(n,s)}^n-x_s^n) (b_s(x_s)-b_s( x_{\kappa(n,s)}^n) )ds \notag
\\
&+ 2 E\int_{t_0}^{t\wedge \tau_{nR}} I_{\{C(R) \leq f(R)\}}(x_{\kappa(n,s)}^n-x_s^n) (b_s( x_{\kappa(n,s)}^n)-b_s^n( x_{\kappa(n,s)}^n) )ds \notag
\end{align}
\begin{align*}
&+E\int_{t_0}^{t\wedge \tau_{nR}}I_{\{C(R) \leq f(R)\}} |\sigma_s(x_s)-\sigma_s( x_{\kappa(n,s)}^n)|^2 ds  \notag\\
&+E\int_{t_0}^{t\wedge \tau_{nR}}I_{\{C(R) \leq f(R)\}} |\sigma_s( x_{\kappa(n,s)}^n)-\sigma_s^n( x_{\kappa(n,s)}^n)|^2 ds  \notag
\\
&+E\int_{t_0}^{t\wedge \tau_{nR}} I_{\{C(R) \leq f(R)\}}\int_{Z} |\gamma_s(x_s,z)-\gamma_s( x_{\kappa(n,s)}^n,z)|^{2} \nu(dz) ds \notag
\\
&+E\int_{t_0}^{t\wedge \tau_{nR}}I_{\{C(R) \leq f(R)\}} \int_{Z} |\gamma_s( x_{\kappa(n,s)}^n,z)-\gamma_s^n( x_{\kappa(n,s)}^n,z)|^{2} \nu(dz) ds \notag
\\
&+2E\int_{t_0}^{t\wedge \tau_{nR}}I_{\{C(R) \leq f(R)\}} (\sigma_s(x_s)-\sigma_s( x_{\kappa(n,s)}^n)) \notag
\\
& \qquad \times (\sigma_s( x_{\kappa(n,s)}^n)-\sigma_s^n( x_{\kappa(n,s)}^n)) ds \notag
\\
&+2E\int_{t_0}^{t\wedge \tau_{nR}} I_{\{C(R) \leq f(R)\}}\int_{Z} (\gamma_s(x_s,z)-\gamma_s( x_{\kappa(n,s)}^n,z)) \notag
\\
& \qquad \times (\gamma_s( x_{\kappa(n,s)}^n,z)-\gamma_s^n( x_{\kappa(n,s)}^n,z))\nu(dz) ds \notag
\end{align*}
for any $t \in [t_0, t_1]$. By using Assumption A-\ref{as:convergence:localmonotonicity},  Schwarz's inequality and H\"older's inequality, one obtains the following estimates,
\begin{align}
E|x_{t\wedge \tau_{nR}}&-x_{t\wedge \tau_{nR}}^n|^{2}I_{\{C(R) \leq f(R)\}} \leq  E|x_{t_0}-x_{t_0}^n|^{2} \notag
\\
&+ E\int_{t_0}^{t\wedge \tau_{nR}} I_{\{C(R) \leq f(R)\}} C(R)|x_s-x_{\kappa(n,s)}^n|^2ds \notag
\\
&+8 RE\int_{t_0}^{t\wedge \tau_{nR}}I_{\{C(R) \leq f(R)\}}  |b_s( x_{\kappa(n,s)}^n)-b_s^n( x_{\kappa(n,s)}^n) | ds \notag
\\
&+ 4 E\int_{t_0}^{t\wedge \tau_{nR}} I_{\{C(R) \leq f(R)\}}C(R)|x_s^n-x_{\kappa(n,s)}^n|ds \notag
\\
&+E\int_{t_0}^{t\wedge \tau_{nR}}I_{\{C(R) \leq f(R)\}} |\sigma_s( x_{\kappa(n,s)}^n)-\sigma_s^n( x_{\kappa(n,s)}^n)|^2 ds  \notag
\\
&+E\int_{t_0}^{t\wedge \tau_{nR}}I_{\{C(R) \leq f(R)\}} \int_{Z} |\gamma_s( x_{\kappa(n,s)}^n,z)-\gamma_s^n( x_{\kappa(n,s)}^n,z)|^{2} \nu(dz) ds \notag
\\
&+2E\int_{t_0}^{t\wedge \tau_{nR}}I_{\{C(R) \leq f(R)\}} (2+|\sigma_s(x_s)|^2+|\sigma_s( x_{\kappa(n,s)}^n)|^2) \notag
\\
& \qquad \times |\sigma_s( x_{\kappa(n,s)}^n)-\sigma_s^n( x_{\kappa(n,s)}^n)| ds \notag
\\
&+2E\int_{t_0}^{t\wedge \tau_{nR}} I_{\{C(R) \leq f(R)\}}\Big(\int_{Z} |\gamma_s(x_s,z)|^2 \nu(dz)\Big)^\frac{1}{2} \notag
\\
& \qquad \times \Big(\int_{Z}|\gamma_s( x_{\kappa(n,s)}^n,z)-\gamma_s^n( x_{\kappa(n,s)}^n,z)|^2\nu(dz) \Big)^\frac{1}{2}ds \notag
\\
& +2E\int_{t_0}^{t\wedge \tau_{nR}} I_{\{C(R) \leq f(R)\}}\Big(\int_{Z}|\gamma_s( x_{\kappa(n,s)}^n,z)|^2\nu(dz)\Big)^\frac{1}{2} \notag
\end{align}
\begin{align*}
& \qquad \times \Big(\int_{Z}|\gamma_s( x_{\kappa(n,s)}^n,z)-\gamma_s^n( x_{\kappa(n,s)}^n,z)|^2\nu(dz) \Big)^\frac{1}{2}ds \notag
\end{align*}
which further implies due to Remarks [\ref{rem:rc:local:bound:sig:gam}, \ref{rem:convergence:siggam:con}] that for $u \in [t_0,t_1]$,
\begin{align}
\sup_{t_0 \leq t \leq u}&E|x_{t\wedge \tau_{nR}}-x_{t\wedge \tau_{nR}}^n|^{2} I_{\{C(R) \leq f(R)\}} \leq E|x_{t_0}-x_{t_0}^n|^{2} \notag
\\
&+   2f(R) \int_{t_0}^{u} \sup_{t_0 \leq r \leq s}E|x_{r\wedge \tau_{nR}}-x_{r\wedge \tau_{nR}}^n|^{2}I_{\{C(R) \leq f(R)\}} ds \notag
\\
&+2E\int_{t_0}^{t_1} I_{\{t_0 \leq s \leq \tau_{nR}\}}I_{\{C(R) \leq f(R)\}}C(R)|x_s^n-x_{\kappa(n,s)}^n|^2ds \notag
\\
&+8 RE\int_{t_0}^{t_1}I_{\{t_0 \leq s \leq \tau_{nR}\}}I_{\{C(R) \leq f(R)\}}  |b_s( x_{\kappa(n,s)}^n)-b_s^n( x_{\kappa(n,s)}^n) | ds \notag
\\
&+ 4 E\int_{t_0}^{t_1} I_{\{t_0 \leq s \leq \tau_{nR}\}}I_{\{C(R) \leq f(R)\}}C(R)|x_s^n-x_{\kappa(n,s)}^n|ds \notag
\\
&+E\int_{t_0}^{t_1}I_{\{t_0 \leq s \leq \tau_{nR}\}}I_{\{C(R) \leq f(R)\}} |\sigma_s( x_{\kappa(n,s)}^n)-\sigma_s^n( x_{\kappa(n,s)}^n)|^2 ds  \notag
\\
&+E\int_{t_0}^{t_1}I_{\{t_0 \leq s \leq \tau_{nR}\}}I_{\{C(R) \leq f(R)\}} \int_{Z} |\gamma_s( x_{\kappa(n,s)}^n,z)-\gamma_s^n( x_{\kappa(n,s)}^n,z)|^{2} \nu(dz) ds \notag
\\
&+4E\int_{t_0}^{t_1}I_{\{t_0 \leq s \leq \tau_{nR}\}}I_{\{C(R) \leq f(R)\}} (C(R)+1)  |\sigma_s( x_{\kappa(n,s)}^n)-\sigma_s^n( x_{\kappa(n,s)}^n)| ds \notag
\\
& +2 E\int_{t_0}^{t_1} I_{\{t_0 \leq s \leq \tau_{nR}\}} I_{\{C(R) \leq f(R)\}}\sqrt{C(R)} \notag
\\
& \qquad \times \Big(\int_{Z}|\gamma_s( x_{\kappa(n,s)}^n,z)-\gamma_s^n( x_{\kappa(n,s)}^n,z)|^2\nu(dz) \Big)^\frac{1}{2}ds<\infty \notag
\end{align}
for any $u \in [t_0, t_1]$. On using Gronwall's inequality, the following estimates are obtained,
\begin{align*}
&\sup_{t_0 \leq t \leq t_1}E|x_{t\wedge \tau_{nR}}-x_{t\wedge \tau_{nR}}^n|^{2}I_{\{C(R) \leq f(R)\}} \leq  \exp(f(R)) \Big\{E|x_{t_0}-x_{t_0}^n|^{2} \notag
\\
&+f(R)E\int_{t_0}^{t_1} I_{\{t_0 \leq s \leq \tau_{nR}\}}|x_s^n-x_{\kappa(n,s)}^n|^2ds\notag
\\
&+8 RE\int_{t_0}^{t_1}I_{\{t_0 \leq s \leq \tau_{nR}\}}I_{\{C(R) \leq f(R)\}}  |b_s( x_{\kappa(n,s)}^n)-b_s^n( x_{\kappa(n,s)}^n) | ds \notag
\\
&+ 8f(R) E\int_{t_0}^{t_1} I_{\{t_0 \leq s \leq \tau_{nR}\}}|x_s^n-x_{\kappa(n,s)}^n|ds \notag
\\
&+E\int_{t_0}^{t_1}I_{\{t_0 \leq s \leq \tau_{nR}\}}I_{\{C(R) \leq f(R)\}} |\sigma_s( x_{\kappa(n,s)}^n)-\sigma_s^n( x_{\kappa(n,s)}^n)|^2 ds  \notag
\\
&+E\int_{t_0}^{t_1}I_{\{t_0 \leq s \leq \tau_{nR}\}}I_{\{C(R) \leq f(R)\}} \int_{Z} |\gamma_s( x_{\kappa(n,s)}^n,z)-\gamma_s^n( x_{\kappa(n,s)}^n,z)|^{2} \nu(dz) ds \notag
\end{align*}
\begin{align*}
&+4(f(R)+1) E\int_{t_0}^{t_1}I_{\{t_0 \leq s \leq \tau_{nR}\}}I_{\{C(R) \leq f(R)\}}  |\sigma_s( x_{\kappa(n,s)}^n)-\sigma_s^n( x_{\kappa(n,s)}^n)| ds \notag
\\
& +2\sqrt{f(R)} E\int_{t_0}^{t_1} I_{\{t_0 \leq s \leq \tau_{nR}\}}I_{\{C(R) \leq f(R)\}} \notag
\\
& \qquad \times \Big(\int_{Z}|\gamma_s( x_{\kappa(n,s)}^n,z)-\gamma_s^n( x_{\kappa(n,s)}^n,z)|^2\nu(dz) \Big)^\frac{1}{2}ds \Big\} \notag
\end{align*}
for every $R>0$. Notice that Assumptions A-\ref{as:sderc:initial:value}, B-\ref{as:esrc:initial} and AB-\ref{as:convergence:initial} imply $E|x_{t_0}-x_{t_0}^n|^2 \to 0$ as $n \to \infty$. Hence, on using Corollary \ref{cor:convergence:onestep} and Assumption AB-\ref{as:convergence:con}, one obtains
\begin{align*}
\lim_{n \to \infty}\sup_{t_0 \leq t \leq t_1}&E|x_{t\wedge \tau_{nR}}-x_{t\wedge \tau_{nR}}^n|^{2}I_{\{C(R) \leq f(R)\}} = 0
\end{align*}
i.e. $T_2\to 0$ for every $R>0$. Further, for any given $\epsilon$, one chooses $R>0$ sufficiently large so that $T_1 < \epsilon/2$ (as it is assumed that $\lim_{R\to\infty}P(C(R)>f(R))=0$) and also $n$ large enough so that $T_2 < \epsilon/2$. As a consequence, one obtains
\begin{align*}
\lim_{n \to \infty}\sup_{t_0 \leq t \leq t_1}E|x_t-x_t^n|^2=0
\end{align*}
which implies that the sequence $\{|x_t-x_t^n|\}_{n \in \mathbb{N}}$ converges to $0$ in probability uniformly in  $t$. Moreover, by taking into consideration Lemmas [\ref{lem:mb:rc}, \ref{lem:esrc:momentbound}], the desired result follows.
\end{proof}
We make the following observations.
%---------------------------------------------------------
\begin{rem} \label{rem:rate:sig:gam:plygrowth}
Due to Assumptions B-\ref{as:esrc:coercivity} and B-\ref{as:rate:rc:bn:polygrowth}, there exist constants $L>0$, $\chi>0$ and a sequence of $\mathscr{F}_{t_0}$-measurable random variables $\{\mathcal{M}_n\}_{n \in \mathbb{N}} \in l_{\infty}({L}^{\frac{p_0}{2}}(\Omega))$ such that, for every $n \in \mathbb{N}$,
 \begin{align*}
|\sigma_t^n(x)|^2  \leq L({M}^n+|x|^{\chi+2})
\end{align*}
almost surely for any $t \in [t_0, t_1]$  and $x \in \mathbb{R}^d$.
\end{rem}
%--------------------------------------------------------
\begin{rem} \label{rem:rate:rc:plolipschitz:sigma:gamma}
Due to Assumptions A-\ref{as:rate:rc:monotonicity} and A-\ref{as:rate:rc:polylipschitz}, there exist constants $L>0$, $\chi>0$ and  $C>0$ such that
\begin{align*}
|\sigma_t(x)-\sigma_t(\bar{x})|^2  \leq C(1+|x|^\chi+|\bar{x}|^\chi)|x-\bar{x}|^2
\end{align*}
almost surely for any $t \in [t_0, t_1]$ and $x, \bar{x} \in \mathbb{R}^d$.
\end{rem}

For the proof of Theorem \ref{thm:rate:rc}, the following lemma is needed.

\begin{lem} \label{lem:rate:rc:one-step}
Let Assumptions B-\ref{as:esrc:initial} to B-\ref{as:rate:rc:bn:polygrowth} be satisfied. Then for any $\rho \in [2,2p_0/(\chi+2)]$, the following holds,
\begin{align*}
\sup_{t_0 \leq t \leq t_1}E|x_t^n-x_{\kappa(n,t)}^n|^\rho \leq Kn^{-1}
\end{align*}
for every $n \in \mathbb{N}$, where  $K$ is a positive constant that does not depend on $n$.
\end{lem}
%--------------------------------------------------------
\begin{proof} By using equation \eqref{eq:esrc}, one obtains
\begin{align*}
E|x_t^n-x_{\kappa(n,t)}^n|^\rho &\leq K E\Big|\int^t_{\kappa(n,t)}b_s^n(x_{\kappa(n,s)}^n)ds\Big|^\rho+K E\Big|\int^t_{\kappa(n,t)}\sigma_s^n(x_{\kappa(n,s)}^n)dw_s\Big|^\rho
\\
&+K E\Big|\int^t_{\kappa(n,t)}\int_Z \gamma_s^n(x_{\kappa(n,s)}^n,z) \tilde{N}(ds, dz)\Big|^\rho
\end{align*}
which on the application of H\"older's inequality and an elementary inequality of stochastic integral yields,
\begin{align*}
E|x_t^n&-x_{\kappa(n,t)}^n|^\rho \leq K n^{-\rho+1}E\int^t_{\kappa(n,t)}|b_s^n(x_{\kappa(n,s)}^n)|^\rho ds
\\
&+K n^{-\frac{\rho}{2}+1}E\int^t_{\kappa(n,t)}|\sigma_s^n(x_{\kappa(n,s)}^n)|^\rho ds
\\
&+ K n^{-\frac{\rho}{2}+1} E\int^t_{\kappa(n,t)}\Big(\int_Z |\gamma_s^n(x_{\kappa(n,s)}^n,z)|^2 \nu(dz)\Big)^\frac{\rho}{2} ds
\end{align*}
\begin{align*}
&+K E\int^t_{\kappa(n,t)}\int_Z |\gamma_s^n(x_{\kappa(n,s)}^n, z)|^\rho \nu( dz) ds
\end{align*}
for any $t \in [t_0,t_1]$. Hence, Assumptions  B-\ref{as:esrc:gamma:growth}, B-\ref{as:rate:rc:bn:polygrowth}, Remark \ref{rem:rate:sig:gam:plygrowth} and Lemma \ref{lem:esrc:momentbound} complete the proof.
\end{proof}
%--------------------------------------------------------
\begin{proof}[\textit{\textbf{Proof of Theorem \ref{thm:rate:rc}}}]
By the application of It\^o's formula for equation \eqref{eq:rc:x-xn},
\begin{align} \label{eq:rate:rc:ito}
|x_t&-x_t^n|^{p} = |x_{t_0}-x_{t_0}^n|^{p} \notag
\\
&+ {p} \int_{t_0}^{t} |x_s-x_s^n|^{p-2} (x_s-x_s^n) (b_s( x_s)-b_s^n( x_{\kappa(n,s)}^n)) ds \notag
\\
&+ p\int_{t_0}^{t} |x_s-x_s^n|^{p-2} (x_s-x_s^n)(\sigma_s( x_s)- \sigma_s^n( x_{\kappa(n,s)}^n))dw_s  \notag
\\
& + \frac{p(p-2)}{2} \int_{t_0}^{t} |x_s-x_s^n|^{p-4}|(\sigma_s( x_s)- \sigma_s^n( x_{\kappa(n,s)}^n))^* (x_s-x_s^n)|^2ds \notag
\\
&+\frac{p}{2}\int_{t_0}^{t} |x_s-x_s^n|^{p-2}|\sigma_s(x_s)- \sigma_s^n( x_{\kappa(n,s)}^n)|^2 ds  \notag
\\
&+  p\int_{t_0}^{t} \int_{Z} |x_s-x_s^n|^{p-2} (x_s-x_s^n) (\gamma_s( x_{s},z)-\gamma^n( x_{\kappa(n,s)}^n,z))    \tilde N(ds,dz) \notag
\\
+\int_{t_0}^{t} &\int_{Z}\{ |x_s-x_s^n+\gamma_s( x_{s},z)-\gamma^n( x_{\kappa(n,s)}^n,z)|^{p}-|x_s-x_s^n|^{p} \notag
\\
&-p|x_s-x_s^n|^{p-2} (x_s-x_s^n)(\gamma_s( x_{s},z)-\gamma^n( x_{\kappa(n,s)}^n,z)) \}N(ds,dz)
\end{align}
almost surely for any $t \in [t_0, t_1]$. One uses the formula for the remainder for the last term on the right hand side of the above equation along with the Schwarz inequality and obtains,
\begin{align}
E|x_t&-x_t^n|^{p} \leq  E|x_{t_0}-x_{t_0}^n|^{p} \notag
\\
&+ p E\int_{t_0}^{t} |x_s-x_s^n|^{p-2}  (x_s-x_s^n) (b_s( x_s)-b_s^n( x_{\kappa(n,s)}^n))ds \notag
\\
&+ \frac{p(p-1)}{2} E\int_{t_0}^{t} |x_s-x_s^n|^{p-2}|\sigma_s( x_s)- \sigma_s^n( x_{\kappa(n,s)}^n)|^2 ds \notag
\\
& +K E\int_{t_0}^{t} \int_{Z} |x_s-x_s^n|^{p-2} |\gamma_s( x_{s},z)-\gamma^n( x_{\kappa(n,s)}^n,z)|^{2} \nu(dz) ds \notag
%\\
%& +K E\int_{t_0}^{t} \int_{Z} |x_s-x_s^n|^{p-3} |\gamma_s( x_{s},z)-\gamma^n( x_{\kappa(n,s)}^n,z)|^{3} \nu(dz) ds \notag
\\
&+KE\int_{t_0}^{t} \int_{Z}|\gamma_s( x_{s},z)-\gamma^n( x_{\kappa(n,s)}^n,z)|^p \nu(dz) ds \notag
\end{align}
for any $t \in [t_0, t_1]$. The above can further be written as,
\begin{align}
&E|x_t-x_t^n|^{p} \leq  E|x_{t_0}-x_{t_0}^n|^{p} \notag
\\
&+ \frac{p}{2} E\int_{t_0}^{t} |x_s-x_s^n|^{p-2} \{2 (x_s-x_s^n) (b_s( x_s)-b_s( x_s^n))  \notag
\\
& \qquad+(p-1)|\sigma_s(x_s)-\sigma_s(x_s^n)|^2 \notag
\\
&\qquad +2(p-1)(\sigma_s(x_s)-\sigma_s(x_s^n))(\sigma_s(x_s^n)-\sigma_s^n(x_{\kappa(n,s)}^n))\}ds\notag
%\\
%&\qquad+(p-1)\int_Z|\gamma_s(x_s,z)-\gamma_s(x_s^n,z)|^2 \nu(dz) \notag
%\\
%&\qquad +2(p-1)\int_Z(\gamma_s(x_s,z)-\gamma_s(x_s^n,z))(\gamma_s(x_s^n,z)-\gamma_s^n(x_{\kappa(n,s)}^n,z))\nu(dz)\Big\}ds\notag
\\
&+ p E\int_{t_0}^{t} |x_s-x_s^n|^{p-2}  (x_s-x_s^n) (b_s( x_s^n)-b_s^n( x_{\kappa(n,s)}^n))ds \notag
\\
&+ \frac{p(p-1)}{2} E\int_{t_0}^{t} |x_s-x_s^n|^{p-2}|\sigma_s( x_s^n)- \sigma_s^n( x_{\kappa(n,s)}^n)|^2 ds \notag
\\
& + K E\int_{t_0}^{t} \int_{Z} |x_s-x_s^n|^{p-2} |\gamma_s( x_{s},z)-\gamma^n( x_{\kappa(n,s)}^n,z)|^{2} \nu(dz) ds \notag
%\\
%& +K E\int_{t_0}^{t} \int_{Z} |x_s-x_s^n|^{p-3} |\gamma_s( x_{s},z)-\gamma^n( x_{\kappa(n,s)}^n,z)|^{3} \nu(dz) ds \notag
\\
&+K E\int_{t_0}^{t} \int_{Z}|\gamma_s( x_{s},z)-\gamma^n( x_{\kappa(n,s)}^n,z)|^p \nu(dz) ds \label{eq:rate:p1}
\end{align}
for any $t \in t\in [t_0, t_1]$. For the second term on the right hand side of the above inequality, one uses Young's inequality, $2ab \leq a^2/(2\epsilon)+\epsilon b^2/2$, with $\epsilon=(p-1)/(2(p_1-p))$ (since $p<p_1$) to obtain the following estimates,
\begin{align*}
(p-1)&|\sigma_s(x_s)-\sigma_s(x_s^n)|^2 +2(p-1)(\sigma_s(x_s)-\sigma_s(x_s^n))(\sigma_s(x_s^n)-\sigma_s^n(x_{\kappa(n,s)}^n))
\\
&\leq (p-1)|\sigma_s(x_s)-\sigma_s(x_s^n)|^2+(p-1)\frac{p_1-p}{p-1}|\sigma_s(x_s)-\sigma_s(x_s^n)|^2
\\
& \qquad+\frac{(p-1)^2}{4(p_1-p)}|\sigma_s(x_s^n)-\sigma_s^n(x_{\kappa(n,s)}^n)|^2
\\
&= (p_1-1)|\sigma_s(x_s)-\sigma_s(x_s^n)|^2+K|\sigma_s(x_s^n)-\sigma_s^n(x_{\kappa(n,s)}^n)|^2
\end{align*}
which on substituting in the right side of \eqref{eq:rate:p1} gives
\begin{align}
E|x_t&-x_t^n|^{p} \leq  E|x_{t_0}-x_{t_0}^n|^{p} \notag
\\
&+ \frac{p}{2} E\int_{t_0}^{t} |x_s-x_s^n|^{p-2} \big\{2 (x_s-x_s^n) (b_s( x_s)-b_s( x_s^n)) \notag
\\
&\qquad +(p_1-1)|\sigma_s(x_s)-\sigma_s(x_s^n)|^2\}ds \notag
%\\
%&\qquad +(p_1-1)\int_Z|\gamma_s(x_s)-\gamma_s(x_s^n)|^2 \nu(dz)\}ds \notag
\\
&+ p E\int_{t_0}^{t} |x_s-x_s^n|^{p-2}  (x_s-x_s^n) (b_s( x_s^n)-b_s^n( x_{\kappa(n,s)}^n))ds \notag
\\
&+ K E\int_{t_0}^{t} |x_s-x_s^n|^{p-2}|\sigma_s( x_s^n)- \sigma_s^n( x_{\kappa(n,s)}^n)|^2 ds \notag
\\
& +K E\int_{t_0}^{t} |x_s-x_s^n|^{p-2}\int_{Z}  |\gamma_s( x_{s},z)-\gamma( x_{s}^n,z)|^{2} \nu(dz) ds \notag
\\
& +K E\int_{t_0}^{t} |x_s-x_s^n|^{p-2}\int_{Z}  |\gamma_s( x_{s}^n,z)-\gamma^n( x_{\kappa(n,s)}^n,z)|^{2} \nu(dz) ds \notag
%\\
%& +K E\int_{t_0}^{t} |x_s-x_s^n|^{p-3} \int_{Z}  |\gamma_s( x_{s},z)-\gamma( x_s^n,z)|^{3} \nu(dz) ds \notag
%\\
%& +K E\int_{t_0}^{t} |x_s-x_s^n|^{p-3} \int_{Z}  |\gamma_s( x_{s}^n,z)-\gamma^n( x_{\kappa(n,s)}^n,z)|^{3} \nu(dz) ds \notag
\\
&+K E\int_{t_0}^{t} \int_{Z}|\gamma_s( x_{s},z)-\gamma( x_s^n,z)|^p \nu(dz) ds \notag
\\
&+K E\int_{t_0}^{t} \int_{Z}|\gamma_s( x_{s}^n,z)-\gamma^n( x_{\kappa(n,s)}^n,z)|^p \nu(dz) ds \notag
\end{align}
which on the application of Assumptions A-\ref{as:rate:rc:monotonicity}, A-\ref{as:rate:rc:gamma:lipschitz}, Schwarz inequality and Young's inequality yields,
\begin{align*}
E|x_t&-x_t^n|^{p} \leq  E|x_{t_0}-x_{t_0}^n|^{p} + K \int_{t_0}^{t} E|x_s-x_s^n|^{p} ds \notag
\\
&+ K E\int_{t_0}^{t} |b_s( x_s^n)-b_s( x_{\kappa(n,s)}^n)|^p ds \notag
\\
&+ K E\int_{t_0}^{t} |b_s( x_{\kappa(n,s)}^n)-b_s^n( x_{\kappa(n,s)}^n)|^p ds \notag
\\
&+ K E\int_{t_0}^{t} |\sigma_s( x_s^n)- \sigma_s( x_{\kappa(n,s)}^n)|^p ds \notag
\\
&+ K E\int_{t_0}^{t} |\sigma_s( x_{\kappa(n,s)}^n)- \sigma_s^n( x_{\kappa(n,s)}^n)|^p ds \notag
\\
& +K E\int_{t_0}^{t} \Big(\int_{Z}  |\gamma_s( x_{s}^n,z)-\gamma( x_{\kappa(n,s)}^n,z)|^{2} \nu(dz)\Big)^\frac{p}{2} ds \notag
\\
& +K E\int_{t_0}^{t} \Big(\int_{Z}  |\gamma( x_{\kappa(n,s)}^n,z)-\gamma^n( x_{\kappa(n,s)}^n,z)|^{2} \nu(dz)\Big)^\frac{p}{2} ds \notag
%\\
%& +K E\int_{t_0}^{t} \Big( \int_{Z}  |\gamma( x_s^n,z)-\gamma( x_{\kappa(n,s)}^n,z)|^{3} \nu(dz)\Big)^\frac{p}{3} ds \notag
%\\
%& +K E\int_{t_0}^{t} \Big( \int_{Z}  |\gamma( x_{\kappa(n,s)}^n,z)-\gamma^n( x_{\kappa(n,s)}^n,z)|^{3} \nu(dz)\Big)^\frac{p}{3} ds \notag
\\
&+K E\int_{t_0}^{t} \int_{Z}|\gamma_s( x_{s}^n,z)-\gamma( x_{\kappa(n,s)}^n,z)|^p \nu(dz) ds \notag
\end{align*}
\begin{align}
&+K E\int_{t_0}^{t} \int_{Z}|\gamma( x_{\kappa(n,s)}^n,z)-\gamma^n( x_{\kappa(n,s)}^n,z)|^p \nu(dz) ds \notag
\end{align}
for any $t \in [t_0, t_1]$. By using Remark \ref{rem:rate:rc:plolipschitz:sigma:gamma}, Assumptions A-\ref{as:rate:rc:gamma:lipschitz} and A-\ref{as:rate:rc:polylipschitz}, one gets,
\begin{align*}
E|x_t&-x_t^n|^{p} \leq  E|x_{t_0}-x_{t_0}^n|^{p} + K \int_{t_0}^{t} E|x_s-x_s^n|^{p} ds +K \int_{t_0}^{t_1} E| x_s^n- x_{\kappa(n,s)}^n|^{p} ds\notag
\\
&+ K E\int_{t_0}^{t_1} (1+|x_s^n|^\chi+|x_{\kappa(n,s)}^n|^\chi)^p|x_s^n- x_{\kappa(n,s)}^n|^p ds \notag
\\
&+ K E\int_{t_0}^{t_1} (1+|x_s^n|^\chi+|x_{\kappa(n,s)}^n|^\chi)^\frac{p}{2}|x_s^n-x_{\kappa(n,s)}^n|^p ds \notag
\\
&+ K E\int_{t_0}^{t_1} |b_s( x_{\kappa(n,s)}^n)-b_s^n( x_{\kappa(n,s)}^n)|^p ds \notag
\\
&+ K E\int_{t_0}^{t_1} |\sigma_s( x_{\kappa(n,s)}^n)- \sigma_s^n( x_{\kappa(n,s)}^n)|^p ds \notag
\\
& +K E\int_{t_0}^{t_1} \Big(\int_{Z}  |\gamma( x_{\kappa(n,s)}^n,z)-\gamma^n( x_{\kappa(n,s)}^n,z)|^{2} \nu(dz)\Big)^\frac{p}{2} ds \notag
%\\
%& +K E\int_{t_0}^{t_1} \Big( \int_{Z}  |\gamma( x_{\kappa(n,s)}^n,z)-\gamma^n( x_{\kappa(n,s)}^n,z)|^{3} \nu(dz)\Big)^\frac{p}{3} ds \notag
\\
&+K E\int_{t_0}^{t_1} \int_{Z}|\gamma( x_{\kappa(n,s)}^n,z)-\gamma^n( x_{\kappa(n,s)}^n,z)|^p \nu(dz) ds \notag
\end{align*}
for any $t \in [t_0, t_1]$. Thus, the application of Gronwall's lemma and H\"older's inequality gives the following estimates,
\begin{align}
&\sup_{t_0 \leq t \leq t_1}E|x_t-x_t^n|^{p} \leq E|x_{t_0}-x_{t_0}^n|^{p} +K \sup_{t_0\leq t \leq t_1} E| x_t^n- x_{\kappa(n,t)}^n|^{p} \notag
\\
&+ K \int_{t_0}^{t_1} \big(E(1+|x_s^n|^\chi+|x_{\kappa(n,s)}^n|^\chi)^{\frac{p(p+\delta)}{\delta}}\big)^\frac{\delta}{p+\delta}\big(E|x_s^n- x_{\kappa(n,s)}^n|^{p+\delta}\big)^{\frac{p}{p+\delta}} ds \notag
\\
&+ K \int_{t_0}^{t_1} \big(E\big(1+|x_s^n|^\chi+|x_{\kappa(n,s)}^n|^\chi\big)^{\frac{p}{2}\frac{p+\delta}{\delta}}\big)^\frac{\delta}{p+\delta}\big(E|x_s^n-x_{\kappa(n,s)}^n|^{p+\delta}\big)^\frac{p}{p+\delta} ds \notag
\\
&+ K E\int_{t_0}^{t_1} |b_s( x_{\kappa(n,s)}^n)-b_s^n( x_{\kappa(n,s)}^n)|^p ds \notag
\\
&+ K E\int_{t_0}^{t_1} |\sigma_s( x_{\kappa(n,s)}^n)- \sigma_s^n( x_{\kappa(n,s)}^n)|^p ds \notag
\\
& +K E\int_{t_0}^{t_1} \Big(\int_{Z}  |\gamma( x_{\kappa(n,s)}^n,z)-\gamma^n( x_{\kappa(n,s)}^n,z)|^{2} \nu(dz)\Big)^\frac{p}{2} ds \notag
%\\
%& +K E\int_{t_0}^{t_1} \Big( \int_{Z}  |\gamma( x_{\kappa(n,s)}^n,z)-\gamma^n( x_{\kappa(n,s)}^n,z)|^{3} \nu(dz)\Big)^\frac{p}{3} ds \notag
\\
&+K E\int_{t_0}^{t_1} \int_{Z}|\gamma( x_{\kappa(n,s)}^n,z)-\gamma^n( x_{\kappa(n,s)}^n,z)|^p \nu(dz) ds \notag
\end{align}
for any $t \in [t_0, t_1]$. The proof is completed by using Lemmas [\ref{lem:esrc:momentbound}, \ref{lem:rate:rc:one-step}] and Assumptions AB-\ref{as:rate:rc:rate} and AB-\ref{as:rate:rc:initial}.
\end{proof}
%--------------------------------------------------------
\section{Numerical Examples}
\subsection*{Example 1} Let us consider the following SDE
\begin{align} \label{eq:num:ex1:sde}
dx_t=(x_t-x_t^3)dt+x_t^2dw_t+x_t\int_{\mathbb{R}} z \tilde{N}(ds,dz)
\end{align}
almost surely for any $t \in [0,1]$ with initial value $x_0=1$. Let us assume that jump intensity is $2$ and mark random variable follows $U(-1/4,1/4)$. The explicit Euler-type scheme is given by
\begin{align} \label{eq:num:ex1:es}
x_{(k+1)h}^n=x_{kh}^n+\frac{x_{kh}^n-(x_{kh}^n)^3}{1+\sqrt{h}|x_{kh}^n|^2} h+\frac{(x_{kh}^n)^2}{1+\sqrt{h}|x_{kh}^n|^2} \Delta w_k+x_{kh}^n \hspace{-2mm}\sum_{i=N(kh)}^{N((k+1)h)} z_i
\end{align}
almost surely for any $k=1,\ldots,n$, where $nh=1$ and $x_0^n=x_0$. In the above, the last term denotes the sum of the jumps in the interval $[kh,(k+1)h]$. As equation \eqref{eq:num:ex1:sde} does not have any explicit solution, the scheme \eqref{eq:num:ex1:es} with step-size $h=2^{-21}$ is treated as the solution of the SDE \eqref{eq:num:ex1:sde}  in the numerical experiment. The number of simulations is $60,000$. The numerical results of Table \ref{tab:num:ex1} and Figure \ref{fig:num:ex1} demonstrate that our numerical findings are consistent with the theoretical results achieved in this paper.
\begin{table}[h]\small
\begin{tabular}{ ||c |c ||c| c|| c| c||} \hline
 $h$ & $\sqrt{E}|x_T-x_T^n|^2$ & $h$ & $\sqrt{E}|x_T-x_T^n|^2$ & $h$ & $\sqrt{E}|x_T-x_T^n|^2$ \\  \hline
$2^{-20}$  & $0.00084487$ & $2^{-15}$  & $0.01090762$ & $2^{-10}$  & $0.04841924$  \\
$2^{-19}$  & $0.00175060$ & $2^{-14}$  & $0.01535016$  & $2^{-9}$   & $0.06225525$ \\
$2^{-18}$  & $0.00297191$ & $2^{-13}$  & $0.02114921$ & $2^{-8}$  & $0.08096656$ \\
$2^{-17}$  & $0.00474922$ & $2^{-12}$ & $0.02838053$ & $2^{-7}$   & $0.10263840$\\
$2^{-16}$  & $0.00744872$ & $2^{-11}$  & $0.03768887$ & $2^{-6}$   & $0.12921045$ \\ \hline
\end{tabular}
\caption{\small $\mathcal{L}^2$-convergence of Euler-type scheme \eqref{eq:num:ex1:es} of SDE \eqref{eq:num:ex1:sde}}
\label{tab:num:ex1}
\end{table}
%--------------------------------------------------------
\begin{figure}
\includegraphics[scale=.40]{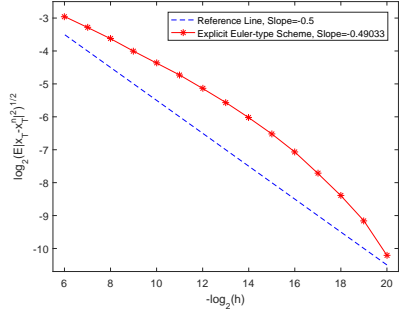}
\caption{\small $\mathcal{L}^2$-convergence of Euler-type scheme \eqref{eq:num:ex1:es} of SDE \eqref{eq:num:ex1:sde}}
\label{fig:num:ex1}
\end{figure}
%--------------------------------------------------------
%--------------------------------------------------------

\end{document}